\documentclass[10pt,twoside]{siamart1116}

\usepackage[english]{babel}
\usepackage{graphicx,epstopdf,epsfig}
\usepackage{amsfonts,epsfig,fancyhdr,graphics, hyperref,amsmath,amssymb}
\usepackage{verbatim}

\usepackage{lscape}

\newcommand{\tluste}[1]{\mbox{\mathversion{bold}$ #1 $}}
\newcommand{\A}[0]{{\tluste{A}}}
\newcommand{\B}[0]{{\tluste{b}}}
\newcommand{\X}[0]{{\tluste{x}}}

\newcommand{\R}[0]{{\mathbb{R}}}

\newcommand{\IR}[0]{{\mathbb{IR}}}

\newcommand{\ol}[1]{\mbox{$\overline{{#1}}$}} 
\newcommand{\ul}[1]{\mbox{$\underline{{#1}}$}}

\newcommand{\magni}{\mathop{\rm mag}\nolimits}

\newcommand{\rad}{\mathop{\rm rad}\nolimits}

\newcommand{\midp}{\mathop{\rm mid}\nolimits}
\newcommand{\www}{\mathop{\rm w}\nolimits}

\def\Mid#1{{#1}_c}		
\def\Rad#1{{#1}_\Delta}		
\def\tluste#1{\protect{\textrm{\boldmath $#1$}}}
\newcommand{\onum}[1]{\mbox{$\overline{{#1}}$}} 
\newcommand{\unum}[1]{\mbox{$\underline{{#1}}$}} 
\newcommand{\inum}[1]{\mbox{$\tluste{#1}$}} 

\newcommand{\jh}[0]{\color{red}}

\newcommand{\HE}{Handling Editor}
\newcommand{\DoS}{Date of Submission}
\newcommand{\DoA}{Date of Acceptance}

\newcommand{\Names}{Jaroslav Hor\'{a}\v{c}ek, Josef Mat\v{e}jka, Milan Hlad\'{i}k}
\newcommand{\Title}{Determinants of interval matrices}

\newtheorem{example}[theorem]{Example}

\renewcommand{\theequation}{\thesection.\arabic{equation}}
\renewtheorem{theorem}{Theorem}[section]
\newtheorem{conjecture}{Conjecture}

\begin{document}

\bibliographystyle{plain}

\setcounter{page}{1}

\thispagestyle{empty}

 \title{\Title\thanks{Received
 by the editors on \DoS.
 Accepted for publication on \DoA. 
 Handling Editor: \HE.}}

\author{
Jaroslav Hor\'{a}\v{c}ek\thanks{Department of Applied Mathematics,
Charles University, Prague, Czech Republic
(horacek@kam.mff.cuni.cz). Supported by a GA\v{C}R grant P403-18-04735S.}
\and 
Milan Hlad\'i{k} \thanks{Department of Applied Mathematics,
	Charles University, Prague, Czech Republic (hladik@kam.mff.cuni.cz). Supported by a GA\v{C}R grant P403-18-04735S. }
\and
Josef Mat\v{e}jka\thanks{Department of Applied Mathematics,
Charles University, Prague, Czech Republic (pipa9b6@gmail.com).}
 }

\markboth{\Names}{\Title}

\maketitle

\begin{abstract}
In this paper we shed more light on determinants of interval matrices. Computing the exact bounds on a determinant of an interval matrix is an NP-hard problem. Therefore, attention is first paid to 
approximations. NP-hardness of both relative and absolute approximation is proved. 
Next, methods computing verified enclosures of interval determinants and their possible combination with preconditioning are discussed. A new method based on Cramer's rule was designed. It returns similar results to the state-of-the-art method, however, it is less consuming regarding computational time. As a byproduct, the Gerschgorin circles were generalized for interval matrices. New  results about classes of interval matrices with polynomially computable tasks related to determinant are proved (symmetric positive definite matrices, class of matrices with identity midpoint matrix, tridiagonal H-matrices). The mentioned methods were exhaustively compared for random general and symmetric matrices. 

\end{abstract}




\section{Introduction} \label{intro-sec}

Interval determinants can be found in various applications. They were used e.g., in \cite{MerDon2006} for testing regularity of inverse Jacobian matrix, in \cite{OetDan2009} for workspace analysis of planar
flexure-jointed mechanism, in \cite{RatRok2003} for computer graphics applications or in \cite{smith1969interval} as a testing tool for Chebyshev systems.

In this work we first address computational properties of determinants of general interval matrices. We are going to prove two new results regarding absolute and relative approximation of interval determinants. Next, we slightly mention known tools that can be used for computing interval determinants -- interval Gaussian elimination, Hadamard inequality and Gerschgorin circles. We introduce our new method based both on Cramer's rule and solving interval linear systems. Regarding symmetric matrices, there are many results about enclosing their eigenvalues and they can be also used for computing interval determinants. All the methods work much better when combined with some kind of preconditioning. We briefly address that topic. 
We also prove that some classes of interval matrices have some tasks related to interval determinant computable in polynomial time (symmetric positive definite matrices, some matrices with identity midpoint matrix, tridiagonal H-matrices).  
At the end we provide thorough numerical testing of the mentioned methods on random general and symmetric interval matrices.

\section{Basic notation and definitions} 
In our work it will be sufficient to deal only with square interval matrices.    
An interval matrix is defined by
$ \A = \{ A \in \R^{n \times n}\ | \ \ul{A} \leq A \leq \ol{A} \}$ for $\ul{A}, \ol{A} \in \R^{n \times n}$ such that $\ul{A} \leq \ol{A}$ (understood component-wise).
To compute with intervals we use the standard interval arithmetic,  for more details on the interval arithmetic see for example \cite{moore2009introduction} or \cite{neumaier1990interval}.

We denote intervals and interval structures in boldface ($\tluste{a}, \A, \B$). Real point matrices and vectors will be denoted in normal case ($A, b$). An interval coefficient of $\A$ lying at the position $(i, j)$ is denoted by $\A_{ij}$.

An interval can be also defined by its midpoint $a_c \in \R$ and radius $a_\Delta \in \R$ as $\tluste{a} = [a_c - a_\Delta, a_c + a_\Delta]$. Interval vectors and matrices are defined similarly. Notation $\midp(\tluste{a}), \rad(\tluste{a})$ can be sometimes used instead of $a_c, a_\Delta $ respectively. 
 The set of all real closed intervals is denoted by  $\IR$ and the set of all square interval matrices of order $n$ is denoted by $\IR^{n \times n}$. When we need (in a proof) open intervals we write them with brackets, i.e. $(\ul{a}, \ol{a})$.

%
%

The magnitude is defined by $\magni(\tluste{a}) = \max ( |\ul{a}|, |\ol{a}|)$ which is sometimes confused with the absolute value $|\tluste{a}| = \{ |a|, a \in \tluste{a} \} $. The width of an interval $\tluste{a}$ is defined by $\www(\tluste{a}) = \ol{a} - \ul{a}$. All these notions can be intuitively defined for vectors, we just use them component-wise. We will also use the interval vector Euclidean norm $\| \X \| = \max \{ \|x\|,   x \in \X \} = \sqrt{\sum \magni (\tluste{x}_i)^2}$. 
 The relation $ \tluste{a} \leq \tluste{b}$ holds when $ \ol{a} \leq \ul{b}$ (similarly for $<$). When we compare two interval structures, the relation is applied component-wise.  
In the following text, $E$ will denote a matrix consisting of ones of a corresponding size. The identity matrix of a corresponding size will be denoted $I$ with $e_i$ denoting its $i$-th column. By $A^+$ we denote the  Moore-Penrose pseudoinverse matrix to $A$ and by $A^{-T}$ we denote the inverse matrix to $A^T$. Spectral radius of $A$ is denoted $\varrho(A)$.
Now, we define the main notion of this work. 

\begin{definition}[Interval determinant]
		Let $\A$ be a square interval matrix, then its interval determinant is defined by
		$$ \det(\A) = \{\det(A), A \in \A \}. $$
\end{definition}

Computing the exact bounds, i.e.,  \emph{hull}, of  $\det(\A)$ is a hard problem. That is why, we are usually satisfied with an enclosure of the interval determinant. Of course, the tighter the better. 

\begin{definition}[Enclosure of interval determinant]
	Let $\A$ be a square interval matrix, then an interval enclosure of its determinant is defined as any $\tluste{d} \in \IR$ such that 
	$$ \det(\A) \subseteq \tluste{d}. $$
\end{definition}

\section{What was known before}
As it was said in the introduction, to the best knowledge of ours, there are only a few theoretical results regarding interval determinants.
Some of them can be found in e.g.,  \cite{KreLak1998,Roh1985,rohn1996checking}.
From linearity of a determinant with respect to matrix coefficients we immediately get the fact that the exact bounds on an interval determinant can be computed as minimum and maximum determinant of all $2^{n^2}$
possible "edge" matrices of $\A$.

\begin{proposition}
	For a given square interval matrix $\A$ the interval determinant can be obtained as 
	$$ \det(\A) = [\min(S), \max(S)], \ \textrm{where} \ S = \{ \det(A), \ \forall i,j \ A_{ij} = \ul{A}_{ij} \ \textrm{or} \ A_{ij} = \ol{A}_{ij}\}.$$
\end{proposition}

%
	
\noindent Regarding complexity of determinant computation we have the following 	theorem \cite{KreLak1998,rohn1996checking}.
	\begin{theorem}
	Computing the either of the exact bounds $\ul{\det(\A)}$ and $\ol{\det(\A)}$ of the matrix $$\A = \left[ A - E, A + E \right], $$ 
	where $A$ is rational nonnegative is NP-hard.

	\end{theorem} 
%

\section{Approximations}
In the end of the previous section we saw that the problem of computing the exact bounds of an interval determinant is generally 
an NP-hard problem. One can at least hope for having some approximation algorithms. Unfortunately, we prove that this is not the case, neither for relative nor for absolute approximation.

\begin{theorem}[Relative approximation]
	Let $\A$ be an interval matrix with $A_c$ nonnegative positive definite matrix and $A_\Delta = E$.  Let $\varepsilon$ be arbitrary such that $0 < \varepsilon < 1$.  If there exists a polynomial time algorithm returning $\left[ \ul{a}, \ol{a} \right] $ such that 
	$$\det(\A) \subseteq \left[ \ul{a}, \ol{a} \right] \subseteq \left[ 1 - \varepsilon, 1 + \varepsilon\right] \cdot \det(\A), $$
	then P = NP.
\end{theorem}
\begin{proof}
From \cite{rohn1996checking} we use the fact that  for a rational nonnegative symmetric positive definite matrix $A$,  checking whether the interval matrix $\A = \left[ A - E, A + E\right]$ is regular (every $A \in \A$ is regular) is a coNP-complete problem.

We show that if such algorithm existed, it would decide whether a given interval matrix is regular. 
For a regular interval matrix we must have $\ul{\det(\A)} > 0$ or $\ol{\det(\A)} < 0$.
If $\ul{\det(\A)} > 0$ then, from the second inclusion  $\ul{a} \geq (1-\varepsilon) \cdot \ul{\det(\A)} > 0$.
On the other hand, if $\ul{a} > 0$ then from the first inclusion $ \ul{\det(\A)} \geq \ul{a} > 0$.
Therefore, we have $\ul{\det(\A)} > 0$ if and only if $\ul{a} > 0$. The corresponding equivalence for $\ol{\det(\A)} < 0$ can be derived in a similar way.
\end{proof}

\begin{theorem}[Absolute approximation]
	Let $A_c$ be a rational positive definite $n \times n$ matrix. Let $\A = \left[A_c - E, A_c + E\right]$ and let $\varepsilon$ be arbitrary such that $0 < \varepsilon$.  If there exists a polynomial time algorithm returning $\left[ \ul{a}, \ol{a} \right] $ such that 
	$$\det(\A) \subseteq \left[ \ul{a}, \ol{a} \right] \subseteq \det(\A) +[-\varepsilon,\varepsilon], $$
	then P = NP.
\end{theorem}
\begin{proof}
Let matrix $A_c$ consist of rational numbers with nominator and denominator representable with $k$ bits (we can take $k$ as the maximum number of bits needed for any nominator or denominator). Then nominators and denominators of coefficients in $A_c -E$ and $A_c + E$ are also representable using $O(k)$ bits. For each row we can multiply these matrices with product of all denominators from both matrices in the corresponding row. Now, each denominator uses still $k$ bits and each nominator uses $O(nk)$ bits. We obtained a new matrix $\A'$. The whole matrix now uses $O(n^3k)$ bits which is polynomial in $n$.  

We only multiplied by nonzero constants therefore the following property is holds 
$$ 0 \notin \det(\A) \Longleftrightarrow 0 \notin \det(\A').$$
After cancellation the new matrix $\A'$ has integer bounds. Its determinant must also have integer bounds. Therefore deciding whether $\A'$ is regular means deciding whether $|\det(\A')| \geq 1$. 
We can multiply one arbitrary row of $\A'$ by $2\varepsilon$ and get a new matrix $\A''$ having $\det(\A'') = 2\varepsilon \det(\A')$.
Now, we can apply the approximation algorithm and compute absolute approximation $[\ul{a}'', \ol{a}'']$ of the determinant of $\A''$.  Let $\det(\A') \geq 1$. Then $\det(\A'') \geq 2\varepsilon$ and the lower bound of absolute approximation is
$$\ul{a}'' \geq \ul{\det(\A'')} - \varepsilon \geq \varepsilon > 0, $$ 
On the other hand, if $\ul{a}'' > 0$ then 
$$ \ul{\det(\A')} / 2\varepsilon =  \ul{\det(\A'')} \geq \ul{a}''> 0.$$
Hence, even $\ul{\det(\A')} > 0$ and since it is an integer it must be greater or equal to 1.
The case of $\det(\A') \leq -1$ is handled similarly.
Therefore, $ 0 \notin \det(\A) \Longleftrightarrow 0 \notin \det(\A') \Longleftrightarrow 0\notin [\ul{a}'', \ol{a}''].$
\end{proof}

\section{Enclosures of determinants -- general case}

\subsection{Gaussian elimination}
To compute a determinant of an interval matrix, we can use the well known Gaussian elimination -- after transforming a matrix to the row echelon form an enclosure of determinant is computed as the product of intervals on the main diagonal. 
For more detailed description of the interval Gaussian elimination see for example \cite{AleHer1983,horacek2013computing,neumaier1990interval}. 
Gaussian elimination is suitable to be used together with some form of preconditioning (more details will be explained in section \ref{sec:precond}). We would recommend the midpoint inverse version as was done in \cite{smith1969interval}.   

%
%

\subsection{Gerschgorin discs}
It is a well known result that a determinant of a real matrix is a product of its eigenvalues. 
To obtain an enclosure of an interval determinant, any method returning enclosures of eigenvalues of a general interval matrix can be used, e.g.,  \cite{Hla2013a, HlaDan2010, Kol2010, May1994}. Here we will employ simple but useful bounds based on the well known Gerschgorin circle theorem.
This classical result claims that for a square real $n \times n$ matrix $A$ its eigenvalues lie inside the circles in complex plane with centers $A_{ii}$ and radius $\sum_{j \neq i} | A_{ij}|$. 
When $\A$ is an interval matrix, to each real matrix  $A \in \A$ there corresponds a set of Gerschgorin discs. Shifting coefficients of $A$ shifts or scales the discs. All discs in all situations are contained inside
discs with centers $ \midp(\A_{ii})$ and radii $\rad(\A_{ii}) + \sum_{j \neq i}  \magni(\A_{ij})$ as depicted in Figure \ref{fig:intervaldisk}.

As in the case of real Gerschgorin discs, it is also well known that in the union of $k$ intersecting circles there somewhere lie $k$ eigenvalues. By intersecting circles we mean that their projection on the horizontal axis is a continuous line. That might complicate the situation a bit. In the intersection of $k$ discs there lie $k$ eigenvalues and their product contributes to the total determinant. That is why, we can deal with each bunch of intersecting discs separately. We compute a verified interval enclosure of a product of $k$ eigenvalues regardless of their position inside this bunch. The computation of the verified enclosure will depend on the number of discs in the bunch (odd/even) and on whether the bunch contains the point 0. In Figure \ref{fig:gersch} all the possible cases and resulting verified enclosures are depicted. 
The resulting determinant will be a product of intervals corresponding to all bunches of intersecting discs.

\begin{figure}[ht]
	\begin{center}
		\epsfig{file=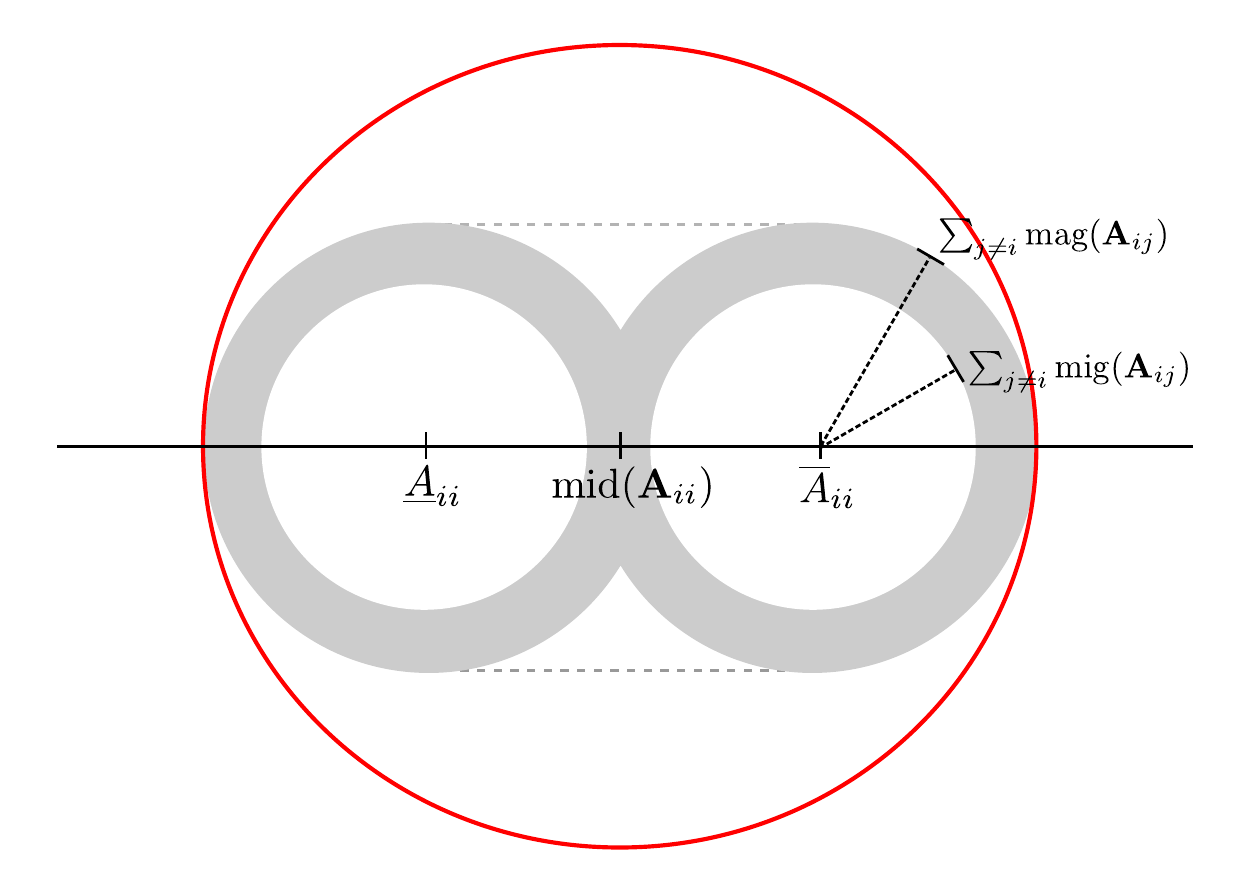,width=7cm,clip=}
		\caption{One interval Gerschgorin disc (large red circle). The grey area mimics the scaling and shifting of a real disc when changing matrix coefficients within their intervals.} 
	\end{center}
	\label{fig:intervaldisk}
\end{figure}

\begin{figure}[ht]
	\begin{center}
		\epsfig{file=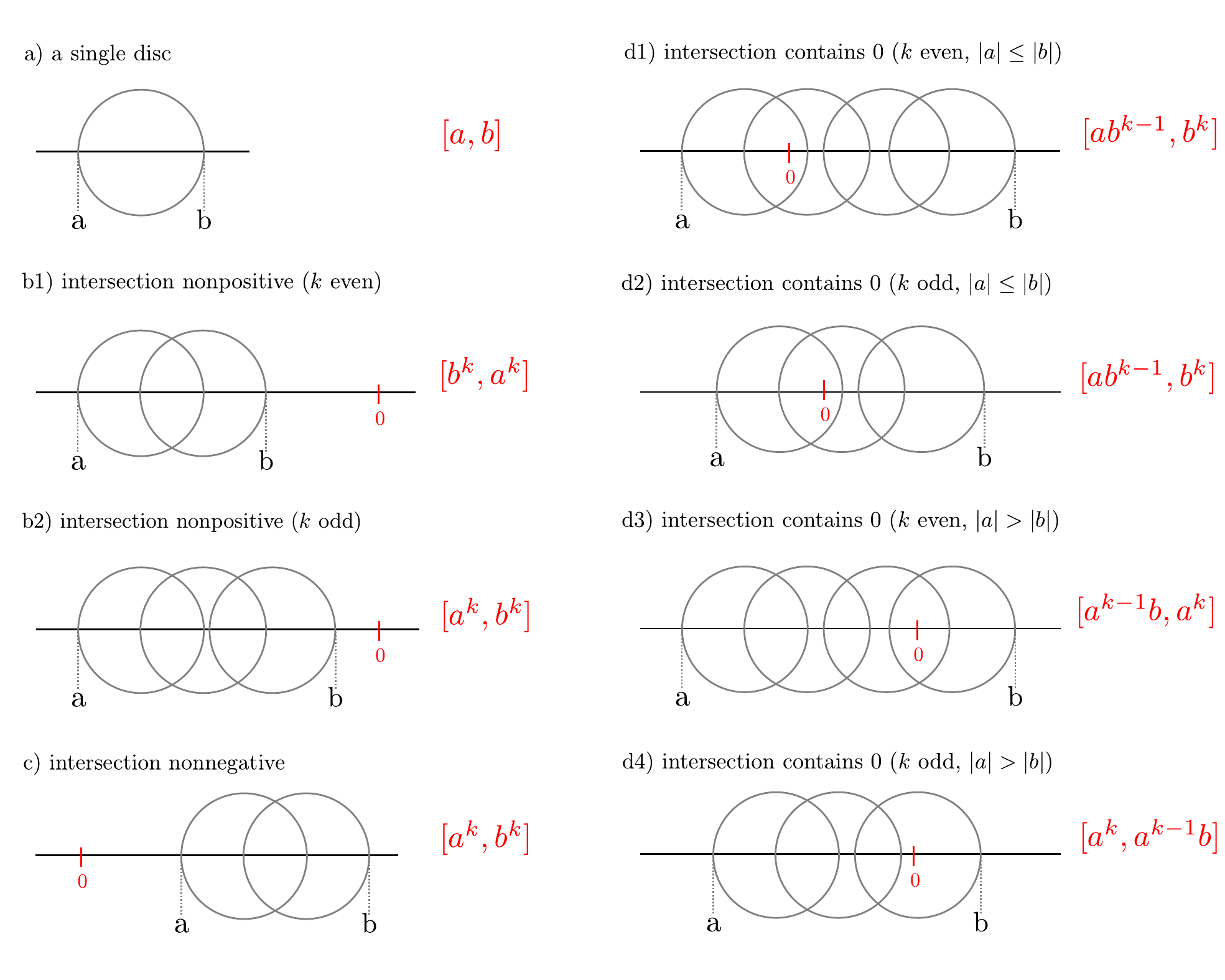,width=150mm,clip=}
\caption{Verified enclosures of product of eigenvalues inside a bunch of $k$ intersecting discs -- all cases.} 
	\end{center}
	\label{fig:gersch}
\end{figure}

The formulas of enclosures are based on the following simple fact. The eigenvalue lying inside an intersection of circles can be real or complex ($c + bi$). In the second case the conjugate complex number $c -bi$ is also an eigenvalue. Their product
$b^2 + c^2$ can be enclosed from above by $a^2$ as depicted in Figure \ref{fig:hypotenuse}. The whole reasoning is based on Pythagorean theorem and geometric properties of hypotenuse.   

\begin{figure}[ht]
	\begin{center}
		\epsfig{file=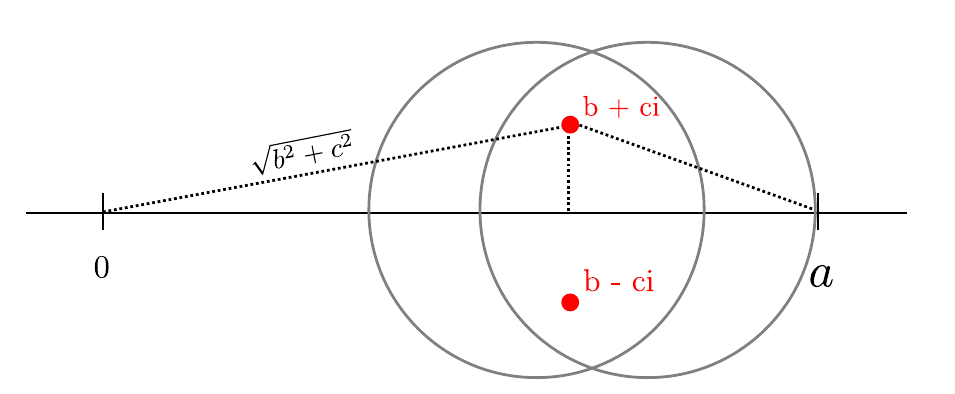,width=7cm,clip=}
		\caption{Enclosing product of two complex eigenvalues.} 
	\end{center}
	\label{fig:hypotenuse}
\end{figure}

%
%

\subsection{Hadamard inequality}
A simple but rather crude enclosure of interval determinant can be obtained by the well known Hadamard inequality. For an $n \times n$ real matrix $A$ we have
$ | \det(A)| \leq \prod_{i=1}^{n} \|A_{*i} \|,$
where $\| A_{*i}\|$ is the Euclidean norm of $i$-th column of $A$. 
This inequality is simply transformable for the interval case. Since the inequality holds for every $A \in \A$ we have

$$ \det(\A) \subseteq \left[-d, +d\right], \ \textrm{where} \ d = \prod_{i=1}^{n} \|\A_{*i} \|.$$
It is a fast and simple method. A drawback is that the obtained enclosure is often quite wide. A second problem is that it is impossible to detect the sign of the determinant, which might be sometimes useful.  

\subsection{Cramer's rule}
Our next method is based on Cramer's rule. It exploits methods for computing enclosure of a solution set of a square interval linear system. There are plenty of such algorithms, i.e., \cite{Hla2014b,moore2009introduction,neumaier1990interval,Rum2010}. Here we use the method "\textbackslash" built in Octave interval package.
When solving a real system $A x = e_1$ using Cramer's rule we obtain 
 $$ \det(A) = \frac{\det(A_{2:n})} {x_1},$$ 
where  $\det(A_{2:n})$ emerges by 
omitting the first row and column from $A$ and $x_1$ is the first coefficient of the solution of $Ax=e_1$. We have reduced our problem of determinant computation to a problem with lower dimension and we can repeat the same procedure iteratively until the determinant in the numerator is easily computable. For an interval matrix $\A$ we actually get
\begin{align}\label{detCramerEncl}
\det(\A) \subseteq \det(\A_{2:n}) / \X_1,
\end{align}
where $\X_1$ is an interval enclosure of the first coefficient of the solution of $\A x = e_1$, computed by some of the cited methods. Notice that we can use arbitrary index $i$ instead of 1. The method works when all enclosures of $\X_1$ in the recursive calls do not
contain 0.

\subsection{Monotonicity checking}

The derivative of a real nonsingular matrix $A\in\R^{n\times n}$ is $\frac{\partial \det(A)}{\partial A}=\det(A)A^{-T}$. Provided the interval matrix $\A$ is regular and $\tluste{B}$ is an interval enclosure for the set $\{A^{-T}\mid A\in\A\}$, then $0\not\in\det(\A)$ and the signs of $\det(\Mid{A})\tluste{B}_{ij}$ give information about monotonicity of the determinant. As long as $0$ is not in the interior of $\tluste{B}_{ij}$, then we can do the following reasoning. If $\det(\Mid{A})\tluste{B}_{ij}$ is a nonnegative interval, then $\det(A)$ is nondecreasing in $A_{ij}$, and hence its minimal value is attained at $A_{ij}=\unum{A}_{ij}$. Similarly for $\det(\Mid{A})\tluste{B}_{ij}$ nonpositive. 

In this way, we split the problem of computing $\det(\A)$ into two sub-problems of computing the lower and upper bounds separately. For each subproblem, we can fix those interval entries of $\A$ at the corresponding lower or upper bounds depending on the signs of $\tluste{B}_{ij}$. This makes the set $\A$ smaller in general. We can repeat this process or call another method for the reduced interval matrix.

Notice that there are classes of interval matrices the determinant of which is automatically monotone. They are called inverse stable \cite{Roh1993c}. Formally, $\A$ is inverse stable if $|A^{-1}|>0$ for each $A\in\A$.
This class also includes interval M-matrices \cite{BarNud1974}, inverse nonnegative \cite{Kut1971} or totally positive matrices \cite{Gar1982} as particular subclasses that are efficiently recognizable; cf.  \cite{Hla2017da}.

\subsection{Preconditioning}
\label{sec:precond}
In the interval case by preconditioning we mean transforming an interval matrix into a better form as an input for further processing. It is generally done by multiplying an interval matrix $\A$ by a real matrix $B$ from left and by a real matrix $C$ from right and we get some new matrix $B \A C$.
Regarding determinants, from properties of the interval arithmetics we easily obtain 
	$ \det(B) \cdot \det(\A) \cdot \det(C) \subseteq \det(B\A C)$ and we will further use the fact
	
	$$ \det(\A)\subseteq \frac{1}{\det(B)\det(C)} \cdot \det(B\A C).$$

There are many possibilities how to choose the matrices $B, C$ for a square interval matrix. 
As in \cite{hansen1967interval}, we can take the midpoint matrix $A_c$ and compute its LU decomposition $PA_c = LU$. When setting $B \approx L^{-1}P, C = I$, we get 
$$\det(\A) \subseteq \frac{1}{\det(P)} \cdot \det(L^{-1} P \A).$$

Another option is using an LDL decomposition. A symmetric positive definite matrix $A$ can be decomposed as $A = LDL^T$, where $L$ is upper triangular with ones on the main diagonal and $D$ being diagonal matrix. By setting  $B \approx L^{-1}, C \approx B^T$ and obtain
$$\det(\A) \subseteq \det(L^{-1} \A L^{-T}).$$

In interval linear system solving, there are various preconditioners utilized depending on criteria used \cite{Hla2016b,Kea1990}.
The most common choice is 
taking $B \approx A^{-1}_c, C = I$ when $A_c$ is regular.  Then
$$ \det(\A) \subseteq \det(A^{-1}_c \A) / \det(A^{-1}_c).$$
Unlike the previous real matrices, the matrix $A^{-1}_c$  does not have to have its determinant equal to $\pm$1. We need to compute a verified determinant of a real matrix. In \cite{ogita2011accurate} there are many variants of algorithms for computation of verified determinants of real matrices. We use the one by Rump \cite{rump2005computer}. 

\section{Enclosures of determinants -- special cases}
Even though we are not going to compare all of the mentioned methods in this section, for the sake of completeness we will mention some cases of matrices that enable the use of other tools. For special classes of interval matrices we prove new results stating that it is possible to compute exact bounds of their determinants in polynomial time. 

\subsection{Symmetric matrices}
Many problems in practical applications are described using symmetric matrices. We specify what we mean by an interval symmetric matrix by the following definition. 
\begin{definition}[Symmetric interval matrix] For a square interval matrix $\A$ we define 
	$$ \A^S = \{ A \in \A, \ A = A^T  \}.$$	
	\end{definition}
Next we define its eigenvalues.
\begin{definition} For a real symmetric matrix $A$ let $\lambda_1 \geq \lambda_2 \geq \ldots \geq \lambda_n$ be its eigenvalues. 
	For $\A^S$ we define its $i$-th set of eigenvalues as $\tluste{\lambda}_i(\A) = \{ \lambda_i(A), \ A \in \A \}.$
	\end{definition}
For symmetric interval matrices there exist various methods to enclose each $i$-th set of eigenvalues. A simple enclosure can be obtained by the following theorem in \cite{HlaDan2010,Roh2012a}.

\begin{theorem}
	\label{th:eig}
	$\tluste{\lambda}_i(\A^S) \subseteq [ \lambda_i(A_c) - \varrho(A_\Delta),  \lambda_i(A_c) + \varrho(A_\Delta)]$
\end{theorem}

There exist various other approaches for computing enclosures of the eigenvalues, including \cite{Kol2006,LenHe2007}. There are several iterative improvement methods \cite{Bea2000,hladik2011filtering}, too. For the exact minimum and maximum extremal eigenvalues, there is a closed-form expression \cite{Her1992}, which is however exponential.

\subsection{Symmetric positive definite matrices}

Let $\A^S$ be a symmetric positive definite matrix, that is, every $A\in\A^S$ is positive definite. Checking positive definiteness of a given symmetric interval matrix is NP-hard \cite{KreLak1998,Roh1994}, but there are various sufficient conditions known \cite{Roh1994b}.

The matrix with maximum determinant can be found by solving the optimization problem
\begin{align*}
\max\ \log\det(A) \mbox{ subject to }A\in\A^S
\end{align*}
since $\log$ is an increasing function and $\det(A)$ is positive on $\A^S$. This is a convex optimization problem that is solvable in polynomial time using interior point methods; see Boyd \& Vandenberghe \cite{BoyVan2004}.
Therefore, we have:

\begin{proposition}
The maximum determinant of a symmetric positive definite matrix is computable in polynomial time.
\end{proposition}

\subsection{Matrices with $A_c=I$}

Preconditioning $\inum{A}$ by $A^{-1}_c$ results in an interval matrix the center of which is the identity matrix $I$. This motivates us to study such matrices more in detail. Suppose that $\inum{A}$ is such that $A_c=I$. Such matrices have very useful properties. For example, solving interval linear systems is a polynomial problem \cite{Roh1993}. Also checking regularity of $\inum{A}$ can be performed effectively just by verifying $\rho(\Rad{A})<1$; see \cite{neumaier1990interval}.

\begin{proposition}
Suppose that $\rho(\Rad{A})<1$. Then the minimum determinant of $\inum{A}$ is attained for $\unum{A}$.
\end{proposition}

\begin{proof}
	
We will proceed by mathematical induction. Case $n=1$ is trivial. 
	
We will proceed by mathematical induction. Case $n=1$ is trivial. For the general case, we express the determinant of $A\in\A$ as in (\ref{detCramerEncl})
\begin{align*}
\det(A) = \det(A_{2:n}) / x_1.
\end{align*}
By induction, the smallest value of $\det(A_{2:n})$  is attained for $A_{2:n}=\unum{A}_{2:n}$. Since $A_c = I$ and $\A$ is regular $\det(A) > 0, \det(A_{2:n}) > 0$, therefore $x_1 > 0$  and as it is the first coefficient of the solution of $A x = e_1$, its largest value is attained for $A=\unum{A}$; see \cite{Roh1993}. Therefore $A=\unum{A}$ simultaneously minimizes the numerator and maximizes the denominator.
\end{proof}

\begin{example} If the condition $\rho(\Rad{A})<1$ does not hold, then the claim is not true in general.  Consider the matrix $\A = [A_c - A_\Delta, A_c + A_\Delta]$ where
$$ \small A_c = \left( \begin{array}{ccc}
1 & 0 & 0\\
0 & 1 & 0 \\
0 & 0 & 1
\end{array} \right), \quad A_\Delta = \left( \begin{array}{ccc}
1 & 1 & 1\\
1 & 1 & 1 \\
1 & 1 & 1
\end{array} \right).$$
We have $\varrho(A_\Delta) = 3$ and $\det(\ul{A}) = -2$, however, $\det(\A) = [-6, 14]$. The minimum bound is attained e.g., for the matrix
$$ \small \left( \begin{array}{ccc}
0 & -1 & 1\\
-1 & 2 & 1 \\
1 & 1 & 2
\end{array} \right).$$  
\end{example}

Computing the maximum determinant of $\inum{A}$ is a more challenging problem, and it is an open question whether is can be done efficiently in polynomial time. 
Obviously, the maximum determinant of $\inum{A}$ is attained for a matrix $A\in\inum{A}$ such that $A_{ii}=\onum{A}_{ii}$ for each $i$. Specifying the off-diagonal entries is, however, not so easy. 

\subsection{Tridiagonal H-matrices}

Consider an interval tridiagonal matrix
\begin{align*}
\small
\A=\begin{pmatrix}
 \inum{a}_1  & \inum{b}_2  &  0  & \dots  & 0 \\
 \inum{c}_2  & \inum{a}_2  & \inum{b}_3 & \ddots & \vdots \\
 0    & \inum{c}_3  & \inum{a}_3 & \ddots & 0 \\
\vdots&\ddots&\ddots& \ddots & \inum{b}_n \\
 0    &\dots & 0   & \inum{c}_n & \inum{a}_n
\end{pmatrix}.
\end{align*}
Suppose that it is an interval H-matrix, which means that each matrix $A\in\inum{A}$ is an H-matrix. Interval H-matrices are easily recognizable, see, e.g., Neumaier \cite{Neu1984,neumaier1990interval}.

Without loss of generality let us assume that the diagonal is positive, that is, $\unum{a}_i>0$ for all $i=1,\dots,n$. Otherwise, we could multiply the corresponding rows by~$-1$.
Recall that the determinant $D_n$ of a real tridiagonal matrix can be computed by a recursive formula as follows
\begin{align*}
D_n = a_n D_{n-1}-b_nc_n D_{n-2}.
\end{align*}
Since $\A$ is an H-matrix with positive diagonal, the values of $D_1,\dots,D_n$ are positive for each $A\in\A$. Hence the largest value of $\det(A)$ is attained at $a_i:=\onum{a}$ and $b_i,c_i$ such that $b_ic_i=\unum{b_i c_i}$. Analogously for the minimal value of $\det(A)$.
Therefore, we have:

\begin{proposition}
Determinants of interval tridiagonal H-matrices are computable in polynomial time.
\end{proposition}
Complexity of the determinant computation for general tridiagonal matrices remains an open problem, si\-milarly as solving an interval system with tridiagonal matrix \cite{KreLak1998}. Nevertheless, not all problems regarding tridiagonal matrices
are open or hard, e.g., checking whether a tridiagonal matrix is regular can be done in polynomial time \cite{bar1996checking}.

\section{Comparison of methods}
In this section the described methods are compared. All these methods were implemented for Octave and its interval package by Oliver Heimlich \cite{heimlich2016gnu}. This package also contains function $\texttt{det}$, which computes an enclosure of the determinant of an interval matrix by LU decomposition, which is basically the same as the already described Gaussian elimination method and that is why we do not explicitly compare the methods against this function. All tests were run on an 8-CPU machine Intel(R) Core(TM) i7-4790K, 4.00GHz. 
 Let us start with general matrices first.

\subsection{General case}
For general matrices the following methods are compared: 

\begin{itemize}
	\item \texttt{GE} - interval Gaussian elimination
	\item \texttt{HAD} - interval Hadamard inequality
	\item \texttt{GERSCH} - interval Gerschgorin circles
	\item \texttt{CRAM} - our method based on Cramer's rule
\end{itemize}	
The suffix "\texttt{inv}" is added when the preconditioning with midpoint inverse was applied and "\texttt{lu}" is added when the preconditioning based on LU decomposition was used.  We use the string \texttt{HULL} to denote the
exact interval determinant.  

\begin{example}
	\label{ex:idet}
	 To obtain a general idea how the methods work, we can use the following example. Let us take the midpoint matrix $A_c$ and inflate it into an interval matrix using two fixed radii of intervals --
$0.1, 0.01$ respectively. 
\begin{align*} \small
A_c=\begin{pmatrix}
	1  & 2  &  3 \\
	4  & 6 &  7\\
	5    & 9  & 8 \\
\end{pmatrix}.
\end{align*}
The resulting enclosures of the interval determinant by all methods are shown in Table \ref{tab:exampledet}.
\end{example}

\begin{table}[h]
	\centering
			\footnotesize
			\renewcommand\arraystretch{1.1} 
	\begin{tabular}{lll}
		\hline
				\hline
		method     & $r=0.1$ & $r=0.01$  \\
		\hline
\texttt{HULL}& [4.060, 14.880] & [8.465, 9.545] \\
\texttt{GE}& [3.000, 21.857]& [8.275, 9.789] \\
\texttt{GEinv} & [3.600, 18.000]& [8.460, 9.560] \\
\texttt{GElu} & [1.440, 22.482]&[8.244, 9.791] \\
\texttt{CRAM} & [-$\infty$, $\infty$]& [8.326, 9.765] \\
\texttt{CRAMinv} & [3.594, 78.230]& [8.460, 9.588] \\
\texttt{CRAMlu} & [-$\infty$, $\infty$]& [8.244, 9.863] \\
\texttt{HAD}& [-526.712, 526.712]& [-493.855, 493.855] \\
\texttt{HADinv}& [-16.801, 16.801]&  [-9.563, 9.563] \\
\texttt{HADlu}& [-35.052, 35.052]& [-27.019, 27.019] \\
\texttt{GERSCH}& [-3132.927, 11089.567]& [-2926.485, 10691.619] \\
\texttt{GERSCHinv} & [-0.000, 72.000]& [6.561, 11.979] \\
\texttt{GERSCHlu}& [-11089.567, 6116.667]& [-10691.619, 5838.410] \\
\hline
		\hline
	\end{tabular}
	\caption[Determinant of A]{The exact interval determinant of the matrix from Example~\ref{ex:idet} and its enclosures computed by various methods. Enclosures bounds are rounded to 3-digits. The fixed radius of intervals is denoted by $r$.}
	\label{tab:exampledet}
\end{table} 

Based on this example it is not worth to test all methods, because some of them do not work well in comparison to others or do not work well without preconditioning. That is why we later test only -- \texttt{GEinv}, \texttt{CRAMinv}, \texttt{HADinv} and \texttt{GERSCHinv}.

We can perceive the method \texttt{GEinv} used in \cite{smith1969interval} as the "state-of-the-art" method. Therefore, every other method will be compared to it. Primarily, for a given matrix $\A$ and a $method()$ we compute the ratio of widths of interval enclosures of $\det(\A)$ computed by both methods as
$$ \textrm{rat}(\A) = \frac{\www ( method(\A) )}{\www( \texttt{GEinv}(\A))}.$$ We test all methods for sizes $n=5,10,15,20,\ldots, 50$ and random interval square matrices with given fixed radii of intervals ($10^{-3}$ or $10^{-5}$). We test on 100 matrices for each size. 
For each size and method  average ratio of computed enclosures, average computation time and its variance is computed. It can happen that an enclosure returned by a method is infinite. Such case is omitted from the computation of average or variance.  

The remaining part to be described is generation of random matrices. First, a random midpoint matrix 
with coefficients uniformly within bounds $\left[-1, 1\right]$ is generated. Then, it is inflated into an interval matrix with intervals having their radius equal to $10^{-3}$ or $10^{-5}$ respectively. 

Let us begin with the average ratios of widths. They are presented Table~\ref{tab:enc}.  When the ratio is a number less then 1000, it is displayed rounded to 2 digits. When it is greater, only the 
approximation $10^x$ is displayed.


	\begin{table}
		\centering
		\footnotesize
		\renewcommand\arraystretch{1.1} 
		\begin{tabular}{c|ccc|ccc }
			\hline
			\hline
			
			size & \texttt{GERSCHinv} &  \texttt{HADinv} & \texttt{CRAMinv} & \texttt{GERSCHinv} &  \texttt{HADinv} & \texttt{CRAMinv}  \\
			\hline
			5 & 8.01& $10^{4}$& 1.00 & 8.88& 41.91& 1.03\\ 
			10 & 19.90& $10^{3}$& 1.00 & 144.46& 16.65& 1.03\\
			15 & 34.96& $10^{3}$& 1.00 & $10^{6}$& 9.04& 1.04\\
			20 & 48.18& $10^{3}$& 1.00 & $10^{10}$& 5.97& 1.04\\
			25 & $10^{10}$& $10^{3}$& 1.00 & $10^{13}$& 4.35& 1.05\\
			30 & 203.06& 251.69& 1.00 & $10^{16}$& 3.71& 1.07\\
			35 & $10^{6}$& 188.74& 1.00 & $10^{19}$& 3.09& 1.06\\
			40 & $10^{14}$& 171.65& 1.00 & $10^{24}$& 2.74& 1.05\\
			45 & $10^{7}$& 128.90& 1.00  & $10^{25}$& 2.28& 1.06\\
			50 & $10^{16}$& 129.55& 1.00 & $10^{28}$& 2.20& 1.07\\

			\hline
			\hline		
		\end{tabular}
		
		\caption{Ratios of enclosures for matrices with fixed radii $10^{-5}$ and $10^{-3}$.}
		\label{tab:enc}
		
	\end{table}

Computation times are displayed in Table~\ref{tab:times}. For each size of matrix the average computation time is displayed; the numbers in brackets are standard deviations. To more clearly see the difference in computation time between the two most efficient methods \texttt{GEinv} and \texttt{CRAMinv} see Figure \ref{fig:besttimes}.

\begin{figure}[ht]
	\begin{center}
		\epsfig{file=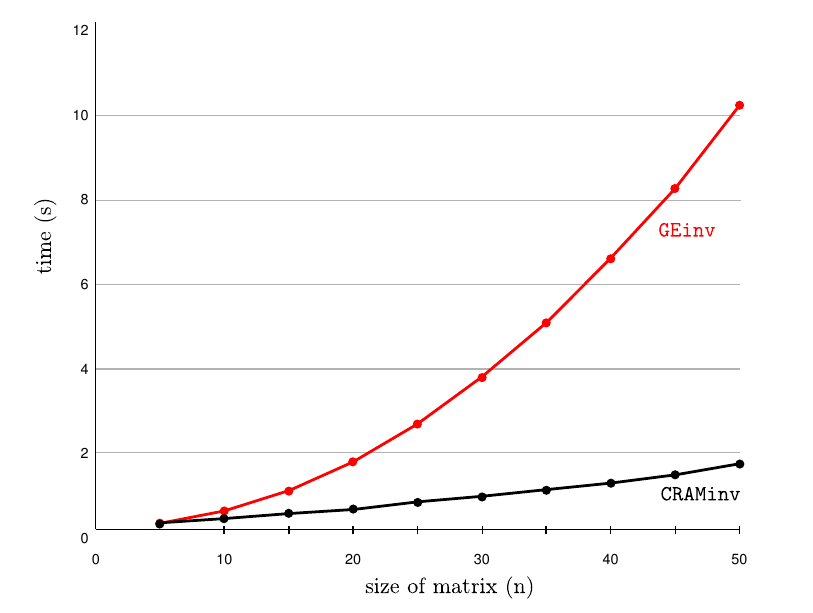,width=9cm,clip=}
		\caption{Comparison between average computation times (in seconds) of \texttt{GEinv} and \texttt{CRAMinv}.} 
	\end{center}
	\label{fig:besttimes}
\end{figure}

\begin{table}
	\centering
	\footnotesize
	\renewcommand\arraystretch{1.1} 
	\begin{tabular}{c|cccc|cccc}
		\hline
		\hline
		
size  & GEinv &  GERSCHinv &  HADinv & CRAMinv & GEinv &  GERSCHinv &  HADinv & CRAMinv\\
\hline
5 & 0.13 & 0.06 & 0.04 & 0.12 & 0.13 & 0.06 & 0.04 & 0.13 \\
& (0.00) & (0.00) & (0.00) & (0.00) & (0.00) & (0.00) & (0.00) & (0.02) \\
\hline
10 & 0.41 & 0.07 & 0.06 & 0.24 10 & 0.40 & 0.07 & 0.06 & 0.25 \\
& (0.00) & (0.00) & (0.00) & (0.00)  & (0.06) & (0.00) & (0.00) & (0.01) \\
\hline
15 & 0.90 & 0.09 & 0.08 & 0.36 15 & 0.91 & 0.09 & 0.08 & 0.39 \\
& (0.04) & (0.00) & (0.00) & (0.00)  & (0.01) & (0.00) & (0.00) & (0.03) \\
\hline
20 & 1.59 & 0.11 & 0.12 & 0.48 20 & 1.51 & 0.11 & 0.12 & 0.54 \\
& (0.01) & (0.00) & (0.00) & (0.01)  & (0.26) & (0.00) & (0.00) & (0.08) \\
\hline
25 & 2.48 & 0.13 & 0.16 & 0.62 25 & 2.41 & 0.13 & 0.16 & 0.73 \\
& (0.07) & (0.00) & (0.00) & (0.03)  & (0.29) & (0.00) & (0.00) & (0.12) \\
\hline
30 & 3.58 & 0.15 & 0.21 & 0.76 30 & 3.47 & 0.15 & 0.21 & 0.92 \\
& (0.02) & (0.00) & (0.00) & (0.01)  & (0.39) & (0.00) & (0.00) & (0.14) \\
\hline
35 & 4.88 & 0.17 & 0.27 & 0.93 35 & 4.59 & 0.17 & 0.27 & 1.09 \\
& (0.03) & (0.00) & (0.00) & (0.02)  & (0.80) & (0.00) & (0.00) & (0.23) \\
\hline
40 & 6.39 & 0.19 & 0.34 & 1.10 40 & 5.77 & 0.19 & 0.34 & 1.25 \\
& (0.03) & (0.00) & (0.00) & (0.04)  & (1.31) & (0.00) & (0.00) & (0.33) \\
\hline
45 & 8.05 & 0.22 & 0.42 & 1.29 45 & 7.34 & 0.22 & 0.42 & 1.48 \\
& (0.59) & (0.00) & (0.00) & (0.09)  & (1.54) & (0.00) & (0.00) & (0.40) \\
\hline
50 & 10.03 & 0.25 & 0.50 & 1.54 50 & 8.77 & 0.25 & 0.50 & 1.68 \\
& (0.04) & (0.00) & (0.00) & (0.06)  & (2.41) & (0.00) & (0.00) & (0.55) \\
\hline
	\hline		
	\end{tabular}
	
	\caption{Times of computation for radii $10^{-5}$ and $10^{-3}$. The plain number is average time (in seconds), the number inside brackets is the standard deviation.}
	\label{tab:times}
	
\end{table}


\subsection{Symmetric matrices}
We repeat the same test procedure with the best methods for interval symmetric matrices. Since these matrices have their eigenvalues real we can add the methods using real bounds on real eigenvalues. 
Symmetric matrices are generated in a similar way as before, only they are shaped to be symmetric. We compare the preconditioned methods with midpoint inverse \texttt{GEinv}, \texttt{GERSCHinv}, \texttt{HADinv} and \texttt{CRAMinv}.
We add one new method \texttt{EIG} based on computation of enclosures of eigenvalues using Theorem \ref{th:eig}. The method \texttt{GEinv} stays the reference method, i.e, we compare all methods with respect to this method.     

The enclosures widths for symmetric matrices are displayed in Table~\ref{tab:encsym}. We can see that as in the general case \texttt{CRAMERinv} does slightly worse than \texttt{GEinv}. Another thing we can see is that \texttt{EIG} is  worse than both  \texttt{CRAMERinv} and \texttt{GEinv}.

	\begin{table}
		\centering
		\footnotesize
		\renewcommand\arraystretch{1.1} 
		\begin{tabular}{c|cccc|cccc}
			\hline
			\hline
		size & \texttt{GERSCHinv} & \texttt{HADinv} &  \texttt{CRAMinv} & \texttt{EIG} & \texttt{GERSCHinv} & \texttt{HADinv} &  \texttt{CRAMinv} & \texttt{EIG}  \\
		\hline
		5 & 7.68& $10^{4}$& 1.00& 2.08 & 7.77& 50.29& 1.01& 2.02\\
		10 & 18.38& $10^{3}$& 1.00& 2.56  & 61.98& 19.22& 1.01& 2.47\\
		15 & 28.38& $10^{3}$& 1.00& 2.99 & $10^{6}$& 11.43& 1.04& 2.73\\
		20 & 44.43& $10^{3}$& 1.00& 3.10 & $10^{7}$& 7.67& 1.03& 2.90\\
		25 & $10^{9}$& $10^{3}$& 1.00& 3.18 & $10^{11}$& 5.70& 1.03& 3.02\\
		30 & 80.43& $10^{3}$& 1.00& 3.33 & $10^{16}$& 4.53& 1.05& 3.10\\
		35 & $10^{5}$& 301.69& 1.00& 3.52 & $10^{18}$& 3.96& 1.04& 3.46\\
		40 & $10^{5}$& 219.13& 1.00& 3.38 & $10^{22}$& 3.41& 1.04& 3.70\\
		45 & $10^{5}$& 183.44& 1.00& 3.48 & $10^{25}$& 2.73& 1.05& 3.65\\
		50 & $10^{3}$& 162.34& 1.00& 3.62  & $10^{26}$& 2.70& 1.04& 4.32\\
			\hline
			\hline		
		\end{tabular}
	
		\caption{Ratios of enclosures for symmetric matrices with radii $10^{-5}$ and $10^{-3}$.}
		\label{tab:encsym}
		
	\end{table}

The computation times are displayed in Table \ref{tab:times5sym}. We can see that \texttt{EIG} shows low computational demands compared to the other methods. One can argue that we can use 
filtering methods to get even tighter enclosures of eigenvalues. However, they work well in specific cases \cite{hladik2011filtering} and the filtering is much more time consuming.

\begin{table}
	\centering
	\footnotesize
	\renewcommand\arraystretch{1.1} 
	\begin{tabular}{cccccc}
		\hline
		\hline
		
		size & \texttt{GEinv} & \texttt{GERSCHinv} & \texttt{HADinv} &  \texttt{CRAMinv} & \texttt{EIG} \\
		\hline
		5 & 0.13 & 0.06 & 0.04 & 0.12 & 0.01\\
		& (0.00) & (0.00) & (0.00) & (0.00) & (0.00)  \\
		\hline
		10 & 0.41 & 0.07 & 0.06 & 0.24 & 0.02\\
		& (0.00) & (0.00) & (0.00) & (0.00) & (0.00)  \\
		\hline
		15 & 0.90 & 0.09 & 0.08 & 0.36 & 0.02\\
		& (0.00) & (0.00) & (0.00) & (0.00) & (0.00)  \\
		\hline
		20 & 1.59 & 0.11 & 0.12 & 0.48 & 0.03\\
		& (0.01) & (0.00) & (0.00) & (0.01) & (0.00)  \\
		\hline
		25 & 2.47 & 0.13 & 0.16 & 0.63 & 0.03\\
		& (0.01) & (0.00) & (0.00) & (0.04) & (0.00)  \\
		\hline
		30 & 3.56 & 0.15 & 0.21 & 0.76 & 0.04\\
		& (0.02) & (0.00) & (0.00) & (0.01) & (0.00)  \\
		\hline
		35 & 4.88 & 0.17 & 0.27 & 0.93 & 0.05\\
		& (0.02) & (0.00) & (0.00) & (0.02) & (0.00)  \\
		\hline
		40 & 6.36 & 0.19 & 0.34 & 1.10 & 0.07\\
		& (0.04) & (0.00) & (0.00) & (0.02) & (0.00)  \\
		\hline
		45 & 8.09 & 0.22 & 0.42 & 1.30 & 0.08\\
		& (0.04) & (0.00) & (0.00) & (0.02) & (0.00)  \\
		\hline
		50 & 9.96 & 0.25 & 0.50 & 1.53 & 0.10\\
		& (0.06) & (0.00) & (0.00) & (0.03) & (0.00)  \\
		\hline	
	\end{tabular}
	
			\caption{Times of computation (in seconds)  for symmetric matrices with radii $10^{-5}$.}
			\label{tab:times5sym}
	
\end{table}


\section{Conclusion}
In the paper we showed that, unfortunately, even approximation of exact bounds of an interval determinant is NP-hard problem (for both relative and absolute approximation). 
On the other hand, we proved that there are some special types of matrices where interval determinant can be computed in polynomial time -- symmetric positive definite, certain matrices with $A_c = I$ or tridiagonal H-matrices.
We discussed four methods \texttt{GE} -- the "state-of-the-art" Gaussian elimination, \texttt{GERSCH} -- our generalized Gerschgorin circles for interval matrices, \texttt{HAD} -- our generalized Hadamard inequality for interval matrices and \texttt{CRAM} -- our designed method based on  Cramer's rule. We introduced a method that can possibly improve an enclosure based on monotonicity checking. All methods combined with preconditioning were tested on random matrices of various sizes.
For interval matrices with radii less than $10^{-3}$ the methods \texttt{GEinv} and \texttt{CRAMinv} return similar results. The larger the intervals the slightly worse \texttt{CRAMinv} becomes. However, its computation time is much more convenient (it is possible to compute a determinant of an interval matrix of order $50$ by \texttt{CRAMinv} at the same cost as an interval matrix of order $20$ by \texttt{GEinv}).  Matrices of order larger than $5$ need some form of preconditioning otherwise \texttt{GE} and \texttt{CRAM} return infinite intervals. In our test cases the \texttt{lu} preconditioning did not prove to be suitable. The methods \texttt{HAD} and \texttt{GERSCH} always return finite intervals, but these intervals can be huge. Both methods work better with the \texttt{inv} preconditioning. The \texttt{HADinv} returns much tighter intervals than \texttt{GERSCH}, however, it can not distinguish the sign of determinant since the enclosure is symmetric around $0$.

The analysed properties of the methods do not change dramatically when dealing with symmetric matrices. The newly added method \texttt{EIG} showed constant and not so huge overestimation and much smaller computation times. The possible improvement of \texttt{EIG} enclosures for symmetric matrices (by applying suitable forms of filtering and eigenvalue enclosures) might be matter of further research. 
There are many more options for future research -- studying various matrix decompositions and preconditioners or studying other special classes of matrices.

\newpage 
\bibliography{literatura}
\bibliographystyle{plain}
\end{document}